\documentclass[12pt]{article}
\usepackage[intlimits]{amsmath}
\usepackage{amsfonts,amssymb,amscd,amsthm,cite}
\usepackage{hyperref}
\input xy
\xyoption{all}

\def\vol{\operatorname{Vol}}

\newcommand{\im}{\mathop{\rm Im}}

\newcommand{\Vect}{\mathop{\rm Vect}}

\def\wf{\operatorname{WF}}

\def\graph{\operatorname{graph}}

\def\ess{\operatorname{ess\; supp}}

\newtheorem{proposition}{Proposition}
\newtheorem{theorem}{Theorem}
\newtheorem{corollary}{Corollary}
\newtheorem{lemma}{Lemma}

\theoremstyle{definition}
\newtheorem{remark}{Remark}
\newtheorem{definition}{Definition}

\def\im{\operatorname{Im}}

\title{Elliptic Operators Associated with Groups of\\ Quantized Canonical Transformations}

\date{}

\author{A. Savin, E. Schrohe, B. Sternin}

\begin{document}

\maketitle

\begin{abstract} 
 Given  a  Lie group $G$ of quantized canonical transformations acting on the space  $L^2(M)$ over a  closed manifold $M$, we define an algebra of so-called $G$-operators on $L^2(M)$.
We show that to $G$-operators we can associate symbols in appropriate crossed products with $G$, introduce a notion of ellipticity and prove the Fredholm property for elliptic elements. 
This framework encompasses many known elliptic theories, for instance, shift operators associated with group actions on $M$, transversal elliptic theory, transversally elliptic pseudodifferential operators on foliations, and Fourier integral operators associated with  coisotropic submanifolds. 
\end{abstract}

\tableofcontents

\section{Introduction}

Let $M$ be a smooth closed manifold and $G$ a discrete group. 
A $G$-operator on $M$ is an operator $D:L^2(M) \to L^2 (M)$ given by a finite sum
\begin{equation}\label{eq-55}
 D=\sum_{g\in G} D_g\Phi_g,
\end{equation}
where the $D_g$ are pseudodifferential operators of order zero on $M$  and $g\mapsto\Phi_g$ defines a representation of $G$ in $\mathcal B L^2(M)$; in fact we shall also consider the more general case of an almost representation,  see Section~\ref{G-operators}.
 
$G$-operators naturally arise, for example,  when $G$ acts on $M$, and the operators $\Phi_g$ are shift operators along the orbits of
the group action: $\Phi_gu(x)=u(g^{-1}(x))$.  Such operators have been studied successfully in the literature, see e.g.~\cite{AnLe1,AnLe2,Con4,CoMo2,NaSaSt17,Per5,SaSt30,SaScSt} and the references cited there; they have interesting applications  in noncommutative geometry, nonlocal problems of mathematical physics and other fields of mathematics. One of the features of  $G$-operators is the fact that their symbols form  essentially noncommutative algebras, crossed products by $G$,  and to understand such symbols one has to use
dynamical properties of the group action on the manifold. 

The aim of this work is to study $G$-operators for which the $\Phi_g$ are given by a representation of $G$ on $\mathcal BL^2(M)$ by {\em quantized canonical transformations}, see e.g.~\cite{Hor4,NOSS1,MSS1,Dui1}. 
Such operators arise  in several recent problems in index theory, see e.g. B\"ar and Strohmaier \cite{BaSt1}\footnote{In fact, the situation considered there is slightly more complicated, since it involves Toeplitz type operators. This case will be addressed in a subsequent publication.}, and in noncommutative geometry, cf. Walters \cite{Walt1}. 
In the simplest case, where $D=\Phi_g$ for a single quantized canonical transformation, we recover the Atiyah-Weinstein index problem  \cite{WE75}; index formulae were given by Epstein and Melrose \cite{EM98} and, in full generality, by Leichtnam, Nest and Tsygan \cite{LNT01}.  
The framework introduced here is a generalization of the case of shift operators, and it is in a certain sense the maximal generalization for which $G$-operators as in \eqref{eq-55} form an algebra with respect to sums and compositions. 
In fact,  a theorem of Duistermaat and Singer \cite{DuSi1} states that the order preserving automorphisms of the algebra $\Psi^*(M)$ of all (classical) pseudodifferential operators on $M$ are given by conjugation by quantized canonical transformations.

As a first result in this article we show in Section 2 that the $G$-operators associated with  (almost) representations of a discrete group $G$ by quantized canonical transformations on  $M$ have symbols in the $C^*$-crossed product $C(S^*M)\rtimes G$ of the algebra of continuous functions on the cosphere bundle of $M$ and $G$, and prove a Fredholm theorem. 

In Section 3 we then study  $G$-operators associated with a Lie group $G$. In this case it is natural to replace the sum in \eqref{eq-55} by an integral over $G$. Again, we define the symbol as an element of a crossed product, introduce ellipticity  and prove a Fredholm  theorem. In contrast to the situation for discrete groups, the symbol for Lie groups is considered on a special 
subspace $T^*_GM\subset T^*M$, which we call the transverse cotangent space. The corresponding crossed product algebra $C(S^*_GM)\rtimes G$
can be viewed  as a replacement for the algebra of functions on the ``bad'' symplectic quotient $T^*_0M//G$.
The key to this result is a careful analysis of the wavefront set of operators of the form $D=\int_G D_g\Phi_g\, dg$, given in Section 4. 

In Section 5 we show how these results encompass a variety of known situations such as the case, where $G$ is a discrete group or a Lie group acting by shifts on $M$, the case of transversally elliptic pseudodifferential operators on foliations, and certain Fourier integral operators associated with  coisotropic submanifolds in the sense of Guillemin and Sternberg. 
 
We finally collect a few basic results on the wave front of oscillatory integrals and quantized canonical transformations in an appendix.

The authors are grateful to Yu.A.~Kordyukov, V.E.~Nazaikinskii and R.~Nest for  useful remarks. 
Supported by DFG (grant SCHR 319 8-1), RFBR (grants  15-01-08392 and 16-01-00373), Ministry of Education and  Science  of the Russian Federation (agreement No. 02.a03.21.0008).

\section{$G$-operators for Discrete Groups}\label{G-operators}

\paragraph{Definition of $G$-operators.}

Let $M$ be a closed smooth manifold and  $G$  a discrete group. Suppose that we are given an almost representation  
\begin{equation}\label{rep1}
 \begin{array}{rcc}
   \rho:G &\longrightarrow & \mathcal{B}L^2(M)\\
    g & \longmapsto & \Phi_g 
 \end{array}
\end{equation}
of this group in   $L^2(M)$ by means of quantized canonical transformations (see  \cite{Dui1,Hor4,MSS1,NOSS1}).
This means that we are given a mapping   $g\mapsto \Phi_g$, which takes elements of   $G$
to quantized canonical transformations  denoted by $\Phi_g$, such that 
\begin{equation}\label{almost1}
 \Phi_e\equiv 1,\qquad \Phi_{gh}\equiv\Phi_g\Phi_h, \qquad g,h\in G.
\end{equation}
In this section   $\equiv$ means equality up to compact operators. The relations  \eqref{almost1} imply ellipticity and the Fredholm property of   $\Phi_g$ for all $g\in G$.

For completeness of the presentation, we review necessary facts about quantized canonical transformations in the appendix.

\begin{definition} A $G$-{\em operator} is an operator that can be written in the form
\begin{equation}\label{gop1}
 D =\sum_g D_g\Phi_g:L^2(M)\longrightarrow L^2(M),
\end{equation}
where $D_g$  are zero-order pseudodifferential operators on $M$, which are nonzero only for a finite set of group elements.  
\end{definition}

\begin{remark}Suppose $G$ acts on $M$ by diffeomorphisms
$g:M\to M$.
Then we can define $\Phi_g$ as the shift operator associated with  $g$:
 $$
  (\Phi_gu)(x)=u(g^{-1}(x)),
 $$
and obtain operators of the form \eqref{gop1},  see e.g.~\cite{AnLe1,AnLe2,NaSaSt17}).    
\end{remark}

\begin{remark}The operators of the form  $D+K$, where $D$ is as in \eqref{gop1} and $K$ is compact on $L^2(M)$, form an algebra, i.e., the sum  and the composition of $G$-operators is (up to compact operators) a $G$-operator as well. 
For the sum this is obvious, while for the composition it follows from  \eqref{almost1} and the fact that the conjugation 
$\Phi_g A \Phi_g^{-1}$ of a  pseudodifferential operator $A$ by a quantized canonical transformation $\Phi_g$ is a pseudodifferential operator whose symbol satisfies   
\begin{equation}\label{ego1}
  \sigma(\Phi_g A \Phi_g^{-1})(x,\xi)=\sigma(A)(g^{-1}(x,\xi)),\quad (x,\xi)\in T^*_0M,
\end{equation}
where $T^*_0M$ is the cotangent bundle of $M$ with the zero section deleted, $\sigma(A)$ is the symbol of $A$ and $g: T^*_0M\to T^*_0M$ is the canonical
transformation.
\end{remark}  
  
\paragraph{Symbols of $G$-operators.}  
  
The aim of the remaining part of this section is to define the symbol of $G$-operators and obtain a Fredholm theorem. 
For the case of infinite groups we use the terminology and methods of the theory of    $C^*$-algebras and their crossed products (see \cite{AnLe1}).  To apply this theory, we pass from  \eqref{rep1} to a unitary representation. 
\begin{proposition}\label{pp32}
Consider the almost representation \eqref{rep1}. Then the mapping
 \begin{equation}\label{rep3}
 \begin{array}{ccc}
   G &\longrightarrow & \mathcal{B}L^2(M) \vspace{2mm}\\
    g & \longmapsto & \widetilde{\Phi}_g = (\Phi_g\Phi_g^*)^{-1/2}\Phi_g,
 \end{array}
\end{equation}
which takes an operator to the unitary part in its polar decomposition, is a  unitary almost representation.  
Moreover, the operator $\widetilde{\Phi}_g$ is also a quantized canonical transformation for $g$.   
\end{proposition}
For simplicity we write here $A^{-1}$ for an inverse of an operator $A$ modulo compact operators.
\begin{proof}
The composition $\Phi_g\Phi_g^*$ is associated with the identity canonical transformation and hence a pseudodifferential operator. 
Moreover, it is nonnegative. Therefore,   $(\Phi_g\Phi_g^*)^{-1/2}$ is well defined modulo compact operators and  a pseudodifferential operator.  
 
Obviously, $\widetilde{\Phi}_g$ is almost unitary.  Let us show that the mapping  \eqref{rep3}
is indeed an almost representation, i.e., 
\begin{equation}\label{q3}
 \widetilde\Phi_{g_1}\widetilde\Phi_{g_2}\equiv\widetilde\Phi_{g_1g_2} 
\end{equation}
modulo compact operators for $g_1,g_2\in G$. Indeed, the left-hand side in \eqref{q3} is equal to 
\begin{equation}\label{q3a}
 \widetilde\Phi_{g_1}\widetilde\Phi_{g_2}\equiv (\Phi_{g_1}\Phi_{g_1}^*)^{-1/2}\Phi_{g_1} (\Phi_{g_2}\Phi_{g_2}^*)^{-1/2} \Phi_{g_1}^{-1}
 \Phi_{g_1g_2}=: A\Phi_{g_1g_2},
\end{equation}
where, according to \eqref{ego1}, the symbol of the pseudodifferential operator $A$ is equal to 
$$
\sigma(A)=  \sigma(\Phi_{g_1}\Phi_{g_1}^*) ^{-1/2} g_1^{-1*}(\sigma(\Phi_{g_2}\Phi_{g_2}^*)^{-1/2}).
$$
On the right-hand side of \eqref{q3} 
\begin{equation}\label{q3b}
 \widetilde\Phi_{g_1g_2} \equiv (\Phi_{g_1}\Phi_{g_2}\Phi_{g_2}^*\Phi_{g_1}^*)^{-1/2}\Phi_{g_1} \Phi_{g_2}=: B\Phi_{g_1g_2},
\end{equation}
where the symbol of the pseudodifferential operator $B$ is equal to
$$
\sigma(B)=  \sigma(\Phi_{g_1}(\Phi_{g_2}\Phi_{g_2}^*)\Phi_{g_1}^{-1}\Phi_{g_1}\Phi_{g_1}^*)^{-1/2} =
(g_1^{-1*}(\sigma(\Phi_{g_2}\Phi_{g_2}^*))\sigma(\Phi_{g_1}\Phi_{g_1}^*))^{-1/2}.
$$
This implies that   $\sigma(A)=\sigma(B)$, so that 
 $A$ and $B$ are equal  modulo compact operators. Now \eqref{q3} follows from  \eqref{q3a} and  \eqref{q3b}, and the proof is complete.
%
\end{proof}

Since  $(\Phi_g\Phi_g^*)^{-1/2}$ is an elliptic pseudodifferential operator, the class of $G$-operators of the form \eqref{gop1} does not change,
if we replace  the operators $\Phi_g$ by the almost unitary operators  $\widetilde{\Phi}_g$. 
Using \eqref{rep3}, we therefore write an arbitrary $G$-operator as
\begin{equation}\label{gop2}
 D =\sum_g \widetilde D_g\widetilde\Phi_g:L^2(M)\longrightarrow L^2(M).
\end{equation}
\begin{definition}
The  {\em symbol} of the $G$-operator \eqref{gop2} is the element
\begin{equation}\label{symb2}
  \sigma(D)=\{\sigma(\widetilde D_g)\}\in C(S^*M)\rtimes  G
\end{equation} 
of the maximal $C^*$-crossed product (see e.g. \cite{Ped1}) of the algebra $C(S^*M)$ of continuous functions on the cosphere bundle   
  $S^*M=T^*_0M/\mathbb{R}_+$ of the manifold and the group $G$, acting on this algebra by automorphisms defined by canonical transformations
$g: S^*M\longrightarrow S^*M$.
We call $D$ elliptic, if its symbol is invertible. 
\end{definition}

\paragraph{The Fredholm theorem.}

\begin{theorem}\label{theo1}
An elliptic  $G$-operator 
$
  D:L^2(M) \longrightarrow L^2(M)
$  
is a Fredholm operator. \end{theorem}
\begin{proof}
1. Consider the  diagram
\begin{equation}\label{diag5}
 \xymatrix{
  C_c(G,\Psi(M)) \ar[d]\ar[r]^\sigma  & C(S^*M)\rtimes G \ar[d]^Q \\
  \mathcal{B}L^2(M) \ar[r]_\pi & \mathcal{B}L^2(M) /\mathcal{K}
  }
\end{equation}
with the ideal $\mathcal K$ of compact operators. Here, the left vertical mapping 
$$
  \{\widetilde D_g\}_{g\in G}\in C_c(G,\Psi(M))\quad\longmapsto\quad  D=\sum_g \widetilde D_g\widetilde \Phi_g: L^2(M)\longrightarrow L^2(M),
$$  
takes a finite collection of operators to the corresponding   $G$-operator, while 
$\sigma$ takes a collection of operators $\{\widetilde D_g\}_{g\in G}$ 
to the collection of the corresponding symbols $\{\sigma(\widetilde D_g)\}_{g\in G}$ and 
$\pi$ denotes the projection onto the Calkin algebra $\mathcal{B}L^2(M) /\mathcal{K}$. Finally, the mapping $Q$ in \eqref{diag5}
is defined on the dense set $\im\sigma$ by 
$$
 Q(\{\sigma(\widetilde D_g)\})=\pi\Bigl(\sum_g \widetilde D_g\widetilde \Phi_g\Bigr).
$$
We claim that $Q$ extends by continuity to the entire crossed product, so that the diagram becomes commutative.
Indeed, by the Gelfand--Naimark theorem, the Calkin algebra  $\mathcal{B}L^2(M) /\mathcal{K}$ can be realized 
as a subalgebra of the bounded operators on a Hilbert space  $H$. In other words, there exists a unital monomorphism of $C^*$-algebras
$$
 i:\mathcal{B}L^2(M) /\mathcal{K}\longrightarrow \mathcal{B}H.
$$
Consider the covariant representation on  $H$ of the $C^*$-dynamical system $(C(S^*M),G)$, which maps an element
$a\in C(S^*M)$ to the operator
$$
 i(a)=i(\pi(A))\in  \mathcal{B}H,
$$
where $A$ is a pseudodifferential operator with symbol   $a$, while to $g\in G$  we assign the unitary operator
$$
 U_g=i(\widetilde\Phi_g)\in  \mathcal{B}H.
$$
It follows from \eqref{ego1} that $$
 U_g i(a ) U_{g}^{-1}= i(g^{-1*}a),
$$
so that the representation is indeed covariant, see e.g.~\cite{Ped1}.

The properties of the maximal crossed product imply that the mapping  
$$
 \{a_g\}_{g\in G}\in C_c(G,C(S^*M)) \longmapsto \sum_g i(a_g)U_g \in \mathcal{B}H,
$$
originally defined for compactly supported functions on the group, extends by continuity to the entire maximal crossed product  $C(S^*M)\rtimes G$. Moreover, it coincides with  $iQ$. This together with the fact that $i$ is a monomorphism, shows that 
  $Q$ is well defined on the maximal crossed product.  


2. If $D$ be an elliptic  $G$-operator, then there exists an inverse   $\sigma(D)^{-1}$ to the symbol of $D$. 
In view of the commutativity of the diagram, we can define the almost-inverse operator $D^{-1}$ by 
$$
 D^{-1}=\pi^{-1}(Q(\sigma(D)^{-1})).
$$
Equivalences $DD^{-1}\equiv 1$, $D^{-1}D\equiv 1$ are easily proved using \eqref{diag5}.

The proof to the theorem is now complete.
\end{proof}

\paragraph{The trajectory symbol.}

The ellipticity condition for $G$-operators in Theorem~\ref{theo1} requires the invertibility of the symbol as an element
of the crossed product and is a condition, which is difficult to check in practice, since the structure of   
crossed products can be quite complicated.  However (under very mild assumptions), this condition can be 
reformulated more explicitly in terms of the so called  
{\em trajectory symbol}.

The trajectory symbol of a $G$-operator, cf. \cite{AnLe2}, is -- just as for  usual pseudodifferential operators -- a function on the cotangent
bundle. However, the trajectory symbol is an operator  acting on the space of functions on the trajectory of the corresponding point.  If we identify 
the trajectory with the group, then the trajectory symbol is given as an operator on $l^2(G)$. 
The trajectory symbol of a pseudodifferential operator is just the operator of multiplication by the value of the symbol of the pseudodifferential
operator at the corresponding point of the orbit, while the symbol of operator   $\widetilde\Phi_g$ is a shift operator on the group. A direct computation shows that the  {\em trajectory symbol}
$$
\sigma(D)(x,\xi):l^2(G)\longrightarrow l^2(G)
$$
of the operator \eqref{gop2} is equal to
\begin{equation}\label{traj1}
 \bigl[\sigma(D)(x,\xi) u\bigr](h)=\sum_g \sigma(\widetilde D_g)(h(x,\xi)) u(g^{-1}h).  
\end{equation}
\begin{remark}
Note that $\sigma(D)(x,\xi)$ is equal to the regular representation at the point $(x,\xi)\in S^*M$ (see \cite{Ped1}, Section 7.7.1) 
of the crossed product $C(S^*M)\rtimes G$  on
$l^2(G)$  and corresponds to the covariant representation
$$
a\in C(S^*M) \longmapsto  a(h(x,\xi)):l^2(G)\longrightarrow l^2(G),
$$
$$
 g\in G \longmapsto L_g u(h)=u(g^{-1}h):l^2(G)\longrightarrow l^2(G).
$$ 
\end{remark}

The results in  \cite{AnLe1},  Theorem 21.2 for an amenable group and \cite{Exe1} in the general case enable us to give the following corollary.
\begin{corollary}
 Suppose that $G$ acts on $S^*M$ amenably,   then the invertibility of the symbol
 $$
   \sigma(D)\in  C(S^*M)\rtimes  G 
 $$
 is equivalent to the invertibility of all the trajectory symbols
 $$
   \sigma(D)(x,\xi):l^2(G)\longrightarrow l^2(G),\qquad (x,\xi)\in S^*M.
 $$ 
\end{corollary}

\section{$G$-operators for Lie Groups}
 
\paragraph{Definition of $G$-operators.}
  
In the sequel we shall denote by $G$ a finite-dimensional, not necessarily compact, Lie group, 
endowed with a left-invariant Haar measure $dg$. 
We write  $\Psi(M)$ for the $C^*$-closure of the algebra of all pseudodifferential operators of order zero on $M$. Recall that the operators in $\Psi(M)$ have symbols in $C(S^*M)$. 
 
Suppose that we are  given a 
unitary representation\footnote{Note that, unlike the situation in the previous section, 
we assume from the start that we are given a unitary representation and not just an ``almost-unitary almost-representation''.}
\begin{equation}\label{rep1q}
 \begin{array}{ccc}
  \rho: G &\longrightarrow & \mathcal{B} L^2(M)\\
    g & \longmapsto & \Phi_g 
 \end{array}
\end{equation} 
on $L^2(M)$  by quantized canonical transformations, see e.g.\  \cite{Dui1,Hor4,MSS1,NOSS1}.
Necessary facts about quantized canonical transformations are recalled in the appendix.  

The representation \eqref{rep1q} is almost never continuous in norm.
Instead, we shall assume that  $\rho$ is continuous in the following sense
\begin{enumerate}
 \item[1)] We are given a smooth action
$$
 \begin{array}{ccc}
   G\times T^*_0M & \longrightarrow & T^*_0 M\\
   (g,m)  & \longmapsto  & g(m)
 \end{array}
$$
of  $G$ on $T^*_0M$ by homogeneous canonical transformations and a smooth family of amplitudes  
$$
 a=\{a_g\}_{g\in G}\in C^\infty_c(G\times S^*M);
$$
\item[2)] We have a decomposition
\begin{equation}\label{deco6}
  \Phi_g=\Phi(g,a_g)+K_g,
\end{equation}
where $\Phi(g,a_g)$ is a fixed quantization of  $g$ and $a_g$,   
while $K_g$ is a norm continuous family of compact operators. For details see the appendix.  
\end{enumerate}

\begin{proposition}\label{pr1}
Let $\{D_g\}_{g\in G}\in C_c(G,\Psi(M))$ be a compactly supported norm continuous function on $G$ with values in the  $C^*$-algebra  $\Psi(M)$.
Then the $G$-operator
 \begin{equation}\label{gop1q}
 D =\int_G  D_g\Phi_g dg:L^2(M)\longrightarrow L^2(M)
\end{equation}
is well-defined  and bounded.
\end{proposition} 

\begin{proof}
1. Let us show that the homomorphism \eqref{rep1q} is strongly continuous, i.e., given $u\in L^2(M)$, 
the function $g\mapsto \Phi_g u$ is norm continuous. Indeed, it suffices to prove this statement for an arbitrary smooth 
function  
$u$ and an operator family $\Phi(g,a_g)$ as in \eqref{deco6}. 
In this case  $\Phi(g,a_g)u$ is defined in local coordinates by an oscillatory integral, which is reduced by a regularization to an 
  absolutely convergent integral. The integrand in the latter integral is continuous in $g$ and the integral is uniformly convergent,
  this implies that the integral is continuous in   $g$. This proves the strong continuity of
  the operator family  $\Phi_g$.

2. Now the action of  $D$ on a function $u\in L^2(M)$  is defined  by the integral
 $$
  Du =\int_G D_g(\Phi_gu)dg,
 $$
 which is norm-convergent in view of the strong continuity of \eqref{rep1q} established above.
The norm of the resulting operator is easily estimated: 
 $$
   \|Du\|\le \int_G \left\|D_g\right\|\| \Phi_g u\|   dg\le \Bigl(\int_G \left\|D_g \right\|   d g\Bigr)\|u\| =C\|u\|.
 $$
 The proof of the proposition is now complete.
\end{proof}

Now that the $G$-operator \eqref{gop1q} is well defined, there naturally arises the problem 
of associating a symbol to such operators and to prove the Fredholm property for elliptic elements. 
It turns out, however,  that such an operator is never Fredholm for a nontrivial group action. In fact,  simple  examples show that the point here is that $G$-operators are smoothing along the orbits of the group
action, see \cite{Ster20} for the case of Lie group actions on the manifold $M$. 
One therefore has to introduce the notion of 
symbol on the space transversal to the orbits, which we call the transverse cotangent space.

\paragraph{The transverse cotangent space and the symbol.}

Denote by $\mathcal G$ the Lie algebra of $G$.
An element  $H\in\mathcal{G}$ defines the Hamiltonian vector field  
$$
V_H=\frac{d}{dt}\Biggr|_{t=0}(\exp tH)^*: C^\infty(T^*_0M) \longrightarrow C^\infty(T^*_0M)
$$ 
on  $T^*_0M$. 
Let $H(x,p)$ be the Hamiltonian function which is homogeneous of degree one in $p$ and generates this vector field.
 
With the action of $G$ on  $T^*_0M$ by homogeneous canonical transformations we   then associate  the following 
closed conical subset of $T^*_0M$:
\begin{equation}\label{tg1}
 T^*_GM=\{(x,p)\in T^*_0M \; |\; H(x,p)=0 \text{ for all } H\in \mathcal{G}\},
\end{equation}
the common zero set of all Hamiltonians generating the group action. We call 
$T^*_GM$ the 
  {\em transverse cotangent space} by analogy with the case of classical transversally elliptic theory \cite{Ati8,Sin2}, where the action is induced by a  group action on  $M$.

\begin{proposition} 
 The subspace $T^*_GM\subset T^*_0M$ has the following properties:
\begin{enumerate}
 \item[1)] In terms of the Hamiltonian vector fields $V_H\in\Vect(T^*_0M)$ it is defined as
\begin{equation}\label{tg1a} 
  T^*_GM=\Bigl\{(x,p)\in T^*_0M \;|\; \omega(\widetilde p,V_H(x,p))=0\quad 
\text{for all  $H\in\mathcal{G} $}\Bigr\},
\end{equation} 
where    $\omega$ denotes the symplectic form on $T^*M$,  and given $p\in T^*_xM$ we define the element $\widetilde p\in
T_{(x,p)}(T^*M)$ equal to a tangent vector to the line $\mathbb{R}p\subset T^*M$.
 \item[2)] The set $T^*_GM\subset T^*_0M$ is  $G$-invariant.
 \item[3)] If the action on $T^*_0M$ is induced by an action on 
 $M$, then the subspace $T^*_GM$ consists of covectors, which vanish on the tangent space to the orbit of $G$ passing through this point and coincides 
 with the space introduced in   {\em \cite{Ati8,Sin2}}.
\end{enumerate}
\end{proposition}
\begin{proof}
 
1.  Let us choose local canonical coordinates $x,p$ on $T^*M$. Then $\omega=dp\wedge dx$, the Hamiltonian vector field with Hamiltonian $H(x,p)$ 
is equal to
$$
V_H= H'_p(x,p) \frac{\partial}{\partial x}-H'_x(x,p) \frac{\partial}{\partial p},
$$ 
and $\widetilde p=p \partial/\partial p$. Substituting these expressions in   $\omega$, we obtain the 
relations
$$
  \omega(\widetilde{p}, V_H)=\omega\left(p \frac\partial{\partial p},H'_p(x,p) \frac{\partial}{\partial x}-H'_x(x,p) \frac{\partial}{\partial p} \right)= pH'_p(x,p)=H(x,p),
$$
which immediately imply the desired equality of the sets  \eqref{tg1} and  \eqref{tg1a}.

2. $G$-invariance of $T^*_GM$ follows from  \eqref{tg1a}.

3. Equation \eqref{tg1a} implies that $\omega(\widetilde{p}, V_H)=  p (V_H)_x $. Hence, the condition $\omega(\widetilde{p}, V_H)=0$
is equivalent to the statement that 
$p$ vanishes on the tangent vectors $(V_H)_x$ for all $H\in \mathcal{G}$.
\end{proof}

Denote by $S^*_GM=T^*_GM/\mathbb{R}_+$ the cosphere space.

\paragraph{The Fredholm theorem.}
\begin{definition}
 {\em The symbol of the $G$-operator} \eqref{gop1q} is the element of the maximal crossed product 
 $$
   \sigma(D)=\{\sigma(D_g)\}\in C(S^*_GM)\rtimes G.
 $$ 
\end{definition}
This crossed product is a nonunital algebra  and therefore does not contain invertible elements. 
To formulate the Fredholm theorem, we adjoin a unit. 
 
\begin{theorem}
 Suppose that the $G$-operator
 $$
  1+D:L^2(M)\longrightarrow L^2(M)
 $$
 is elliptic, i.e.,  its symbol
 $$
   1+\sigma(D)\in (C(S^*_GM)\rtimes  G)^+
 $$
is invertible in the crossed product with adjoint unit. Then $1+D$ is a Fredholm operator.  
\end{theorem}

\begin{proof}
1. The proof of this theorem is based on the following commutative diagram
\begin{equation}\label{diag7}
 \xymatrix{
  C_c(G, C(S^*M)) \ar[dr]_{Q_0} \ar[r] & C(S^*M)\rtimes G \ar[d]^{Q_1} \ar[r] &  C(S^*_G M)\rtimes G \ar[dl]^Q\\
   & \mathcal{B}L^2(M)/\mathcal{K}
 }
\end{equation}
The horizontal arrows in the diagram are given by the natural embedding
of the set of compactly supported functions on the group into the crossed product and the restriction to the 
$G$-invariant subspace $S^*_GM\subset S^*M$, respectively. The mapping 
\begin{equation}\label{muk4}
 Q_0:C_c(G, C(S^*M))\to \mathcal{B}L^2(M)/\mathcal{K}
\end{equation}
takes a collection of symbols  $\{\sigma(D_g)\}_{g\in G}$ to the corresponding 
$G$-operator \eqref{gop1q} defined uniquely up to compact operators. 
The existence of the vertical mapping $Q_1$  is guaranteed by Lemma 1, below,  which will be proven in the next section.
\begin{lemma}\label{lem1}
 The mapping \eqref{muk4} extends by continuity to a $C^*$-algebra homomorphism   denoted by
 $$
  Q_1:C(S^*M)\rtimes G \to \mathcal{B}L^2(M)/\mathcal{K},
 $$
 such that the left triangle in \eqref{diag7} is commutative.
\end{lemma}
We complete the construction of Diagram \eqref{diag7} with the following lemma, whose proof is also given in the next section.
\begin{lemma}\label{lem2}
 There exists a mapping
 $$
  Q:C(S^*_GM)\rtimes G \to \mathcal{B}L^2(M)/\mathcal{K},
 $$
such that the right triangle in \eqref{diag7} is commutative.
\end{lemma} 
This lemma implies that  a  $G$-operator with symbol   vanishing on the subset 
$S^*_GM\subset S^*M$  is compact.

2. Using diagram \eqref{diag7}, we can easily prove the Fredholm theorem. Indeed, consider an elliptic   $G$-operator $1+D$. 
Then its symbol 
$1+\sigma(D)$ is invertible with the inverse symbol denoted by $(1+\sigma(D))^{-1}$. 
It follows $Q((1+\sigma(D))^{-1})$ furnishes an inverse to $1+D$ in the Calkin algebra; in particular, $1+D$ is a Fredholm operator. 
\end{proof}

\paragraph{The trajectory symbol.}

In applications, it is useful to have a trajectory symbol, whose invertibility is easier to check than the invertibility in the maximal crossed
product. 

We define the regular representation  at a point $(x,\xi)\in S^*_GM$  (see \cite{Ped1}, Section 7.7.1) of the crossed product $C(S^*_GM)\rtimes  G$  on
$L^2(G,dg)$  in terms of the covariant representation
$$
a\in C(S^*_GM) \longmapsto  a(h(x,\xi)):L^2(G,dg)\longrightarrow L^2(G,dg),
$$
$$
 g\in G \longmapsto L_g u(h)=u(g^{-1}h):L^2(G,dg)\longrightarrow L^2(G,dg).
$$
This regular representation at a point   $(x,\xi)$ takes the symbol to the element denoted by
$$
\sigma(D)(x,\xi):L^2(G,dg)\longrightarrow L^2(G,dg)
$$
and called the {\em trajectory symbol} of operator $D$. Explicitly, the trajectory symbol of   \eqref{gop1q} 
acts on a function $u(h)\in L^2(G,dg)$ as
$$
 \Bigl[\sigma(D)(x,\xi) u\Bigr](h)=\int_G \sigma(D_g)(h(x,\xi)) u(g^{-1}h)  dg.
$$
\begin{proposition}
 Suppose that the action of  $G$ on $S^*_GM$  is amenable. Then the invertibility of the symbol
 $$
  1+\sigma(D)\in (C(S^*_GM)\rtimes G)^+
 $$
 is equivalent to the invertibility of all the trajectory symbols
 $$
  1+\sigma(D)(x,\xi):L^2(G,dg)\longrightarrow L^2(G,dg), \quad (x,\xi)\in S^*_GM.
 $$
\end{proposition}
\begin{proof}
 This statement follows from the results of Ionescu and Williams \cite{IoWi1} on the structure of primitive ideals of 
 $C^*$-algebras of amenable groupoids and is obtained in~\cite{NiPr1}, Theorem 3.18.
\end{proof}

\section{Proofs of the Auxiliary Results}

\subsection{Proof of Lemma \ref{lem1}}

For the proof, we need a $C^*$-crossed product of the algebra of pseudodifferential operators   $\Psi(M)$ and $G$
acting by automorphisms on this algebra by conjugation
\begin{equation}\label{act1}
 A\in\Psi(M) \longmapsto  \Phi_g A \Phi_g^*\in \Psi(M).
\end{equation}
To have a well-defined crossed product   (see \cite{Ped1}), we need the following proposition.
\begin{proposition}\label{pr2}
 The action \eqref{act1} is strongly continuous, i.e., given a pseudodifferential operator   $A$, the operator-function $\Phi_g A \Phi_g^*$ 
 is norm continuous in $g\in G$. 
\end{proposition}
 
By Proposition \ref{pr2}, the pair $(\Psi(M),G)$ is a $C^*$-dynamical system and, hence, the maximal crossed product denoted
 by  $\Psi(M)\rtimes G$ is well defined. Further, this $C^*$-dynamical  system has a covariant represention in  $L^2(M)$.
 This covariant representation gives a representation
 of the crossed product on  $L^2(M)$, which takes a collection of operators $\{D_g\}_{g\in G}$ to the
  $G$-operator \eqref{gop1q}.

Consider the commutative diagram
\begin{equation}\label{diag1a}
 \xymatrix{
      0 \ar[r]  & \mathcal{K}\rtimes  G \ar[r]\ar[d] &  \Psi(M)\rtimes  G  \ar[d]\ar[r]^{\sigma} & C(S^*M)\rtimes  G \ar[r] & 0\\
      0 \ar[r]  & \mathcal{K} \ar[r]                     &  \mathcal{B}L^2(M)\ar[r]                   & \mathcal{B}L^2(M)/\mathcal{K}\ar[r] & 0
 } 
\end{equation}
Here the vertical mappings take elements of crossed products (i.e., operator functions on the group) to the corresponding $G$-operators on $M$. 
The upper row of the diagram is a component-wise crossed product of the standard row  
$$
 0\longrightarrow \mathcal{K} \longrightarrow \Psi(M) \stackrel\sigma\longrightarrow C(S^*M) \longrightarrow 0
$$
by $G$.  The diagram \eqref{diag1a} is commutative by construction, while its rows are exact  
(for the upper row this follows from the exactness of the maximal crossed product functor   \cite{Wil1}, Proposition 3.19, 
while for the lower row this is obvious).
 
We now define the desired mapping    $Q_1$ by requiring that the diagram, below, 
\begin{equation}\label{diag1b}
 \xymatrix{
      0 \ar[r]  & \mathcal{K}\rtimes  G \ar[r]\ar[d] &  \Psi(M)\rtimes  G  \ar[d]\ar[r]^{\sigma} & C(S^*M)\rtimes  G \ar[r]\ar[d]_{Q_1} & 0\\
      0 \ar[r]  & \mathcal{K} \ar[r]                     &  \mathcal{B}L^2(M)\ar[r]                   & \mathcal{B}L^2(M)/\mathcal{K}\ar[r] & 0
 } 
\end{equation}
is commutative.  
Indeed, given $a\in C(S^*M)\rtimes  G$, the exactness of the upper row shows that $a\in\im \sigma$ and there exists
$\sigma^{-1}(a)\in \Psi(M)\rtimes  G$. We map this element to the corresponding   $G$-operator and then to its image in the Calkin algebra
 $\mathcal{B}L^2(M)/\mathcal{K}$, and denote it by  $Q_1(a)$. 
The mapping $Q_1$ is well-defined since a different choice of  $\sigma^{-1}(a)$
yields an operator, which differs by a compact operator, and hence defines the same element in the Calkin algebra. Moreover, $Q_1$ defines a $C^*$-algebra homomorphism. This completes the proof of Lemma~\ref{lem1}.

\subsection{Proof of Lemma \ref{lem2}}

To prove Lemma~\ref{lem2},  let us compute the wave front set for $G$-operators.

\paragraph{Wave front sets of    $G$-operators.}

Recall (e.g., see \cite{Hor1}) that the wave front set  of an operator
$$
D:C^\infty(M)\longrightarrow \mathcal{D}'(M)
$$ 
is defined in terms of the wave front set of its Schwartz kernel $K_D\in \mathcal{D}'(M\times M)$ as
$$
\wf'(D)=\{(x,p,x',p')\in T^*_0(M\times M) \;|\; (x,p,x',-p')\in \wf(K_D)\}.
$$

\begin{proposition}\label{pwf1}
Given a $G$-operator
\begin{equation}\label{eq-gop1}
  u\longmapsto Du=\int_G D_g\Phi_g u dg,
\end{equation}
where  $\{D_g\}$ is a smooth family of classical pseudodifferential operators on $M$, we have
 \begin{equation}\label{f32}
   \wf'(D)\subset  \left\{(x,p,x',p')\in T^*_0(M\times M)\;\Bigl|\; 
    \begin{array}{l}\exists g\in G: (x,p,x',p')\in \graph g\\
        (x,p), (x',p')\in T^*_GM, (x,p)\in \ess  \sigma_\psi(D_g)
       \end{array}
\right\},
\end{equation}
where $\sigma_\psi(D_g)$ is the complete symbol of pseudodifferential operator $D_g$, while $\ess$ is the smallest closed conical set outside
of which the symbol has order   $-\infty$.
\end{proposition}

\begin{proof}
1. Consider the case, where $u$ depends additionally on $g\in G$. Define the operator 
$$
 u(\cdot,g)\longmapsto D'u=\int_G D_g \Phi_g u(\cdot ,g) dg 
$$
from functions on $M\times G$ to functions on $M$. Since $D=D'\pi^*$, where 
$$
 \pi^*:C^\infty(M)\to C^\infty(M\times G)
$$ 
is the operator of constant extension with respect to   $g$,
we get for the Schwartz kernels
$$
 K_D\in \mathcal{D}'(M\times M),\qquad K_{D'}\in \mathcal{D}'(M\times (M\times G))
$$
the formula
$$
K_D=\pi'_*K_{D'},\qquad \text{where }\pi':M\times (M\times G) \longrightarrow M\times M \text{ is the natural projection}.
$$

2. It suffices to show   \eqref{f32} in the situation, when $D_g$ is identically zero for all $g$ outside a small neighborhood of the identity.  
This follows from the fact that the integral can be written as a sum of integrals over neighborhoods that cover the group. For each integral over a neighborhood of, say,  a point $g_0\in G$, we may take the factor $\Phi_{g_0}$ outside the integral  and use the fact that $\Phi_{g_0}$ 
acts on wave front sets by shifting them along the canonical transformation  $g_0$, together with the $G$-invariance of the set  $T^*_GM$.

3. Applying formula \eqref{f1} in the appendix to the distribution $K_D$, we obtain the following expression for the wave front set of this
distribution  
\begin{equation}\label{f2}
  \wf(K_D)\subset\left\{(x,\Psi'_{x},x',\Psi'_{x'})\in T^*_0(M\times M)\;\Bigr|\; \begin{array}{c}
   \exists g\in G, p'\ne 0: \Psi'_{g,p'}=0, \\ (x,p )\in \ess \sigma_\psi (D_g) 
\end{array}   \right\}.
\end{equation}
Here   the phase function is equal to  $\Psi=S(x,g,p')-p'x'$ (we use the fact that $g$ is in a small neighborhood of the identity,
and  $\Phi_g$ is microlocally defined as in  \eqref{qq1}). Note that the condition $\Psi'_{g}=0$  in \eqref{f2} means that
\begin{equation}\label{cond4}
   \frac{d}{dt}\Bigl|_{t=0 }S(x,g_t,p')=0,   
\end{equation}
where
$$
g_t=g h_t,\qquad h_t(x,p)\equiv (x,p)+t(H'_p,-H'_x)(x,p).
$$
Here and below  $\equiv$ means equality modulo $O(t^2)$, and,  Equation 
  \eqref{cond4} should be satisfied for all Hamiltonians $H\in \mathcal{G}$.

Since $\Psi'_g=S'_g$ and $\Psi'_{p'}=S'_{p'}-x'$,  it follows from \eqref{f2} that
$$
\wf'(K_D)\subset \left\{(x,p,x',p') \Bigl| \exists g: (x,p,x',p')\in\graph g,  \frac{d}{dt}\Bigl|_{t=0 }S(x,g_t,p')=0\right\};
$$
here we used the fact that $S$ is a generating function for the Lagrangian manifold $\graph g$, i.e., it satisfies Equations
 \eqref{phi77} from the appendix. For the computation of the generating function see \eqref{eq-ars1} in the appendix. We have
\begin{multline}\label{der3}
S(x,g_t,p')=p'(g^{-1}_t(x,p_t))_x\equiv p'(h_{-t}g^{-1} (x,p_t))_x\equiv \\
\equiv p'\Bigl[ g^{-1}(x,p_t)-t(H'_p,-H'_x)(g^{-1} (x,p))\Bigr]_x\equiv\\
\equiv
p'x'+t\Bigl[ p'\frac{\partial (g^{-1})_x}{\partial p}\frac{dp_t}{dt}\Bigl|_{t=0 }-p'H'_p(x',p')\Bigr] 
=p'x'-tH(x',p'),
\end{multline}
where we used the homogeneity of   $H(x,p)$ in $p$ and Euler's formula
$
 p'H_p(x',p')=H(x',p'),
$
and also the equality
$$
p'\frac{\partial (g^{-1})_x}{\partial p}=0,
$$
which follows from the homogeneity of the canonical transformation   $g$, see \eqref{hom4} in the appendix.

It follows from  \eqref{der3} that  \eqref{cond4} is equivalent to
$$
H(x',p')=0.
$$
Since $T^*_GM$ is $G$-invariant, this implies also that $H(x,p)=0$, and the proof of the proposition is complete.
\end{proof}

\paragraph{Proof of Lemma~\ref{lem2}.}

1. Let us show that the mapping  $Q_1$ in~\eqref{diag7} is identically zero on the subalgebra
$$
 C_0(S^*M\setminus S^*_GM)\rtimes G \subset  C(S^*M)\rtimes G. 
$$
It suffices to prove this statement on the dense set of smooth families  
$$
 \{\sigma(D_g)\}_{g\in G}\in C^\infty_c(G\times (S^*M\setminus S^*_GM)).
$$
We treat such a smooth family as a family of symbols of pseudodifferential operators  $\{D_g\}_{g\in G}$, 
whose complete symbols vanish identically in a neighborhood of    $S^*_GM$.
Then by Proposition~\ref{pwf1} the wave front set of the $G$-operator \eqref{eq-gop1} is empty, i.e., its Schwartz kernel
is smooth and the operator is smoothing, hence, compact. This gives us zero element in the Calkin algebra. This ends the proof of the statement. 
 
2. Consider the short exact sequence of $C^*$-algebras
$$
 0\longrightarrow  C_0(S^*M\setminus S^*_GM)\rtimes G \longrightarrow  C(S^*M)\rtimes G\longrightarrow 
 C(S^*_G M)\rtimes G
 \longrightarrow 0. 
$$
By Item 1 above, the mapping $Q_1$ is defined on the middle term in this sequence and vanishes
on the ideal  
$C_0(S^*M\setminus S^*_GM)\rtimes G$. Hence, it extends to the quotient, which is isomorphic to  $C(S^*_GM)\rtimes G$.

The proof of Lemma~\ref{lem2} is now complete.
 
 \section{Examples and Remarks}

1. Let $G$ be a discrete and amenable group, acting on $M$.  
We choose 
as unitary operators $\widetilde\Phi_g$ the weighted shift operators
$$
 \widetilde\Phi_g =\left(\frac {g^{-1*}\vol}{\vol}\right)^{1/2} T_g: L^2(M)\longrightarrow L^2(M),\quad
 T_gu(x)=u(g^{-1}x).
$$
Then we obtain as a corollary the Fredholm theorems in   \cite{AnLe1,AnLe2}.

2. If $G$ is a compact Lie group acting on   $M$, then we obtain the Fredholm results  in \cite{SaSt22,Ster20}.

3. If $G$ is a  Lie group acting on $M$ locally-freely, then the
$G$-operators coincide with the transverse pseudodifferential operators   
with respect to the foliation on $M$ defined by the orbits of the group action which were studied by Kordyukov in \cite{Kord3}.

4. In~\cite{GuSt3}, Guillemin and Sternberg introduced  algebras of Fourier integral operators associated with  smooth homogeneous submanifolds    $\Sigma\subset T^*_0M$ such that  
\begin{itemize}
\item $\Sigma$ is coisotropic (this means that for all $\xi\in \Sigma$ we have $(T_\xi\Sigma)^\perp\subset T_\xi\Sigma$,
where  $(T_\xi\Sigma)^\perp$ stands for the orthogonal space with respect to the symplectic form on   $T^*_0M$). 
This condition implies that the family of linear subspaces   $(T_\xi\Sigma)^\perp$, $\xi\in\Sigma$ 
defines a foliation on $\Sigma$, which we denote by  $\mathcal{F}$.
\item The foliation $\mathcal{F}$ has a smooth Hausdorff space of leaves, i.e., there exists a smooth manifold   $S$ and a projection 
$$
 \rho: \Sigma \longrightarrow S,
$$
whose fibers are the leaves of the foliation $\mathcal{F}$.
\item The $\mathbb{R}_+$-action on  $\mathcal{F}$ by dilations extends to a free action on   $S$.
\end{itemize}

Under these assumptions, Guillemin and Sternberg considered the 
algebra of Fourier integral operators associated with the canonical relation   
$$
 \mathcal{C}=\{(\xi,\eta)\in \Sigma\times \Sigma \;|\; \rho(\xi)=\rho(\eta)\}\subset T^*_0M\times T^*_0M.
$$ 

It turns out  that in order to construct this operator algebra, it is not necessary to  require  that the space of leaves of $\mathcal{F}$ is smooth.
In fact, one can drop this assumption by considering the corresponding 
Fourier integral operators as operators associated with the foliation, see \cite{Kord3} in the situation, where the foliation  $\mathcal{F}$ 
is induced by a foliation on the main manifold  $M$, see also a class of pseudodifferential operators in \cite{SaSt11}.

In the present paper, an important role is played by the closed homogeneous subspace  $T^*_GM\subset T^*_0M$,
i.e., the set of common zeroes of the Hamiltonians, which generate the action of the Lie group   $G$ on $T^*_0M$. 
A straightforward computation shows that, whenever $T_GM$ is a smooth submanifold   (e.g., when
the differentials of the Hamiltonians are linearly independent) it is coisotropic  and the theory of Guillemin and Sternberg   applies. 
One can show that the algebra of such Fourier integral operators coincides with the above described algebra of $G$-operators.

\begin{remark}
The space $T^*_GM$ is in general singular and hence, the classical theory of Fourier integral operators 
can not be applied.  Let us give a simple example. On  $\mathbb{R}^2$ with coordinates  $x_1,x_2$ and dual coordinates  $p_1,p_2$,
consider the Hamiltonian action of the group   $G=\mathbb{R}$ generated by the Hamiltonian
$$
 H(x,p)=x_1^2p_1+x_2^2p_2.
$$
The zero   set of this Hamiltonian is obviously not a smooth manifold. Hence, Guillemin and Sternberg's theory does not apply to this Hamiltonian, while
the methods of our paper can be applied.  From this point of view this article can be seen as an extension of their theory 
to singular submanifolds  $\Sigma\subset T^*_0M$. 
\end{remark}

\section{Appendix}

For the sake of completeness we recall some basic facts from the theory of oscillatory integrals and quantized canonical transformations.  

\subsection{Oscillatory Integrals}

\paragraph{Wave front sets of oscillatory integrals.}
We need the following result concerning wave front sets of distributions defined by oscillatory integrals, e.g., see \cite{Dui1}, Theorem 2.2.2. 

For $X\subset \mathbb R^m$ consider the distribution $A\in \mathcal D'(X)$
\begin{equation}\label{eq-63}
u\longmapsto A(u)= \iint e^{i\Psi(x,\theta)}a(x,\theta) u(x)d\theta dx,\qquad u\in  C^\infty_c(X),
\end{equation}
defined in terms of an amplitude
$a(x,\theta)\in S^m(X\times  \mathbb{R}^n)$ and a real phase function  $\Psi(x,\theta)\in C^\infty(X\times (\mathbb{R}^n\setminus 0))$, i.e., $\Psi$ is homogeneous of degree $1$ in $\theta$ and 
satisfies the condition $\Psi'_{x,\theta}\ne 0$ whenever $\Psi=0$.

According to the stationary phase lemma  the wave front set of $A$ satisfies
\begin{equation}\label{dui1}
  \wf(A)\subset \{(x,\Psi'_x)\in T^*_0X \;|\; \exists  \theta\ne 0:  \Psi'_\theta=0\text{ and }(x,\theta)\in\ess\;a\}.
\end{equation}
Here $\ess a\subset  X\times  (\mathbb{R}^n \setminus 0)$ is the minimal closed conical set outside of which   $a\in S^{-\infty}$.

\paragraph{A generalization.}

Let now the phase function and the amplitude  depend smoothly on a parameter   $t\in T$. 
We are interested in distributions $A\in \mathcal{D}'(X)$ of the form
$$
u\longmapsto A(u)= \iiint e^{i\Psi(x,t,\theta)}a(x,t,\theta) u(x)d\theta dx dt,\qquad 
u\in C^\infty_c(X).
$$
\begin{proposition}\label{pr4}
 The wave front set of the distribution $A\in \mathcal{D}'(X)$  satisfies
 \begin{equation}\label{f1a}
  \wf(A)\subset \{(x,\Psi'_{x})\in T^*_0X \;|\;  \exists (\theta,t), \theta\ne 0:\;  \Psi'_{\theta,t}=0, \text{ and }(x,t,\theta)\in\ess\;a\}.
\end{equation}
\end{proposition}
\begin{proof}
Because of the integral over $t$ we cannot apply \eqref{dui1} directly.
Instead, we introduce the auxiliary distribution  $A'\in \mathcal{D}'(X\times T)$ given by 
$$
u\longmapsto A'(u)= \iiint e^{i\Psi(x,t,\theta)}a(x,t,\theta) u(x,t)d\theta dx dt,\qquad u\in C^\infty_c(X\times T).
$$

The latter distribution is well defined, since   $\Psi'_{x,t,\theta}\ne 0$ when $\Psi=0$ (this follows from the nondegeneracy of the phase for each fixed   $t$). According to \eqref{dui1} we get
$$
\wf(A')\subset \{(x,t,\Psi'_{x,t})\in T^*_0(X\times T) \;|\;  \exists \theta\ne 0:\;  \Psi'_\theta=0\text{ and }(x,t,\theta)\in\ess\;a\}
$$

Further, for the projection  $\pi:X\times T\to X$ we have the pushforward mapping for distributions 
$$
\pi_*: \mathcal{D}'(X\times T) \longrightarrow \mathcal{D}'(X)
$$
$$
 \pi_*(A') u=A'(\pi^*u), \qquad u\in C^\infty(X).
$$
Clearly $A=\pi_* (A')$. One the other hand, Proposition 1.3.4 in \cite{Dui1}  shows how the wave front set is transformed
under this pushforward. Namely, 
$$
\wf(\pi_*(A'))\subset \{(x,p)\in T^*_0X \;|\;  \exists t\in T:\; (x,t,\pi^*(p))\in \wf(A')\}
$$
Hence, we obtain the following equality for the wave front set of   $A$:
\begin{equation}\label{f1}
  \wf(A)\subset \{(x,\Psi'_{x})\in T^*_0X \;|\;  \exists (\theta,t), \theta\ne 0:\;  \Psi'_{\theta,t}=0, \text{ and }(x,t,\theta)\in\ess\;a\}.
\end{equation}
The proof of Proposition~\ref{pr4} is complete.
\end{proof}

\subsection{Quantized Canonical Transformations}

\paragraph{Homogeneous canonical transformations.}

Let
$$
 g:T^*_0M\longrightarrow T^*_0M
$$ 
be a homogeneous canonical transformation. It is well known that this condition is equivalent to the requirement that 
  $g$ preserves the canonical $1$-form $pdx$ on $T^*_0M$. In local coordinates  $x,p$, if we write $g$ as
\begin{equation}
(x,p)=(x_1,x_2,...,p_1,p_2,...)\longmapsto (g_x(x,p),g_p(x,p))=(g_{x1},g_{x2},...,g_{p1},g_{p2},\ldots),
\end{equation}
the condition that $g$ preserves $pdx:$ $g^*(pdx)=pdx$ is written as
\begin{equation}\label{hom4}
 \sum_j g_{pj}\frac{\partial g_{xj}}{\partial p_k}=0,\qquad \sum_j g_{pj}\frac{\partial g_{xj}}{\partial x_k}=p_k,
 \quad \forall k.
\end{equation}

Consider the Lagrangian manifold (the graph of $g$)
$$
L=\graph g=\{(gm,m)\in T^*_0M\times T^*_0M | m\in T^*_0M\}.
$$
On the product $T^*_0M\times T^*_0M$, we denote local coordinates as $(x,p,x',p')$.

\begin{proposition}
 Suppose that  $g$ is close to the identity transformation and corresponds to a Hamiltonian flow.  Then on the Lagrangian manifold   $L=\graph g$
 we can choose the canonical coordinates $x,p'$ and in addition the generating function of this manifold is expressed in terms of $g $
 as\footnote{For instance, if $g=id$, then $S(x,p')=xp'.$ }
\begin{equation}\label{eq-ars1}
  S(x,g,p')=p'x':= p' \bigl( (g_x)^{-1}(x,p')\bigr).            
\end{equation}
In other words,  $L$ is defined by the equations
\begin{equation}\label{phi77}
L=\{(x,p,x',p')\in T^*_0M\times T^*_0M \;| \; p=S'_x(x,g,p'),\; x'=S'_{p'}(x,g,p') \}.
\end{equation}
\end{proposition}
\begin{remark}
In \eqref{eq-ars1} $x'$ is calculated as follows: we take the $x$-component $g_x$ of g and then take  the inverse function $(g_x)^{-1}$ with $p'$ fixed. Note that the same result is obtained if we replace this expression
by $(g^{-1}(x,p))_x$.   
\end{remark}
\begin{proof}
 The proof essentially follows from the theory of Hamilton-Jacobi equations, see e.g.~\cite{Arn1}.

1. Suppose that $g(x,p)$ is equal to the Hamiltonian flow for time   $t=1$ for a Hamiltonian  $H(x,p)$ (which is homogeneous of degree one in $p$). In other words, we have 
 $g(x,p)=(x(1),p(1))$, where the functions $(x(t),p(t))$ are defined as the solutions of the Cauchy problem
 $$
  \left\{
     \begin{array}{l}
       \dot x=H'_p,\\
       \dot p=-H'_x,\\
       x(0)=x,p(0)=p.
     \end{array}
  \right.
 $$
 
2.  Consider the Cauchy problem for the Hamilton--Jacobi equation
$$
  \left\{
     \begin{array}{l}
       S'_t+H(x,S'_x)=0,\\
       S|_{t=0}=xp' .
     \end{array}
  \right.
$$
By Hamilton--Jacobi theory (and since the Hamiltonian is independent of $t$) we see that the solution of this
Cauchy problem can be written as  
\begin{equation}\label{eq-s}
 S(x(z,t),t)=S(z,0)+\int(pdx-Hdt)=S(z,0)+0\equiv zp',
\end{equation}
where $x(z,t)$ is the $x$-component  of the solution of the Hamiltonian system
$$
  \left\{
     \begin{array}{l}
       \dot x=H'_p\\
       \dot p=-H'_x\\
       x(0)=z,p(0)=p'.
     \end{array}
  \right.
$$
Now we take $t=1$ in  \eqref{eq-s} and express  $z$ in terms of  $x=x(z,1)=g_x(z,p')$.  Then we see that we obtain precisely the
function \eqref{eq-ars1}.

3. It remains to show that the function   $S(x,g,p')$ is a generating function for the Lagrangian manifold $\graph g$.
This is proved by a direct computation, which is easy to do using  \eqref{hom4}.  

The proof of proposition is now complete. 
\end{proof}

\paragraph{Quantized canonical transformations.}

Let $g:T^*_0M\longrightarrow T^*_0M$ be a homogeneous canonical transformation. 
By a quantized canonical transformation associated with $g$ we mean an arbitrary Fourier integral operator associated with 
the Lagrangian manifold  $L={\rm graph }g\subset T^*_0M\times T^*_0M$.
Detailed expositions of Fourier integral operators can be found in   \cite{Hor4,NOSS1,MSS1,Dui1}.

Let us give an explicit expression for the quantized canonical transformation in the case, when   $g$ 
is equal to a Hamiltonian flow at small times, i.e., the transformation is close to the identity. 
In this case as canonical coordinates on   $L$  we can take 
$(x,p')$ and have explicit description for the generating function $S$, see \eqref{eq-ars1}. Then the quantization of the mapping 
  $g$ is microlocally defined by 
\begin{equation}\label{qq1}
 \Phi(g,a) u(x) =\iint e^{iS(x,g,p')-ip'x'}a(x,g,p')u(x')dp'dx',
\end{equation}
where $a$ is an amplitude function. 
Note that $\Psi=S(x,g,p')-p'x'$ in the latter integral satisfies conditions for phase functions, since $\Psi'_{x'}=-p'\ne 0$ off the zero section $p'=0$,
hence, the operator \eqref{qq1} is well defined, namely, its Schwartz kernel is defined as a distribution equal to this oscillatory integral.

\end{document}